\documentclass[11pt]{amsart}
\usepackage{lmodern}
\usepackage{amsmath, amsthm, amssymb, amsfonts,tikz}
\usepackage[normalem]{ulem}
\usepackage{hyperref}

\usepackage{mathrsfs}

\usepackage{verbatim} 
\usepackage{longtable}

\hypersetup{
    colorlinks=true,
    linkcolor=black,
    filecolor=magenta,      
    urlcolor=cyan,
    pdftitle={Overleaf Example},
    pdfpagemode=FullScreen,
    }

\usepackage{mathtools}

\usetikzlibrary{decorations.pathmorphing}
\tikzset{snake it/.style={decorate, decoration=snake}}

\usepackage{caption}

\usepackage{tikz-cd}
\usetikzlibrary{arrows}

\theoremstyle{plain}
\newtheorem{thm}{Theorem}[section]
\newtheorem{cor}[thm]{Corollary}
\newtheorem{lem}[thm]{Lemma}
\newtheorem{prop}[thm]{Proposition}
\newtheorem{conj}[thm]{Conjecture}

\theoremstyle{definition}

\theoremstyle{remark}
\newtheorem{rmk}[thm]{Remark}

\newcommand{\BC}{{\mathbb{C}}}

\newcommand{\BF}{{\mathbb{F}}}

\newcommand{\BP}{{\mathbb{P}}}
\newcommand{\BQ}{{\mathbb{Q}}}

\newcommand{\BZ}{{\mathbb{Z}}}

\newcommand{\CE}{{\mathcal E}}
\newcommand{\CF}{{\mathcal F}}

\newcommand{\CO}{{\mathcal O}}
\newcommand{\CP}{{\mathcal P}}

\newcommand{\CU}{{\mathcal U}}

\newcommand{\sslash}{\mathbin{/\mkern-6mu/}}

\DeclareFontFamily{OT1}{rsfs}{}
\DeclareFontShape{OT1}{rsfs}{n}{it}{<-> rsfs10}{}
\DeclareMathAlphabet{\curly}{OT1}{rsfs}{n}{it}

\newcommand\Hom{\operatorname{Hom}}

\usepackage{tikz}
\usepackage{lmodern}
\usetikzlibrary{decorations.pathmorphing}

\begin{document}
\title[Perverse, Chern, and refined BPS invariants]{Perverse filtrations, Chern filtrations, and refined BPS invariants for local $\BP^2$}
\date{\today}

\author[Y. Kononov]{Yakov Kononov}
\address{Yale University}
\email{yakov.kononov@yale.edu}

\author[W. Pi]{Weite Pi}
\address{Yale University}
\email{weite.pi@yale.edu}

\author[J. Shen]{Junliang Shen}
\address{Yale University}
\email{junliang.shen@yale.edu}

\begin{abstract}
We explore connections between three structures associated with the cohomology of the moduli of 1-dimensional stable sheaves on $\BP^2$: perverse filtrations, tautological classes, and refined BPS invariants for local $\BP^2$. We formulate the $P=C$ conjecture identifying the perverse filtration with the Chern filtration for the free part of the cohomology. This can be viewed as an analog of de Cataldo--Hausel--Migliorini's $P=W$ conjecture for Hitchin systems. Our conjecture is compatible with the enumerative invariants of local $\BP^2$ calculated by refined Pandharipande--Thomas theory or Nekrasov partition functions. It provides a cohomological lift of a conjectural product formula of the asymptotic refined BPS invariants. We prove the $P=C$ conjecture for degrees $\leq 4$.
\end{abstract}

\maketitle

\setcounter{tocdepth}{1} 

\tableofcontents
\setcounter{section}{-1}

\section{Introduction}


We work over the complex numbers $\BC$.

\subsection{Refined BPS invariants for local $\BP^2$}
Let $X = \mathrm{Tot}(K_{\BP^2})$ be the Calabi--Yau 3-fold given by the total space of the canonical bundle on $\BP^2$, which is referred to as the ``local $\BP^2$''. Let $d\geq 3$ be an integer. Considerations from physics predict that there is an action of $\mathfrak{sl}_2 \times \mathfrak{sl}_2$ on the cohomology of a certain moduli space of $D$-branes supported on a degree $d$ curve in $X$ \cite{GV}; this yields double indexed integral invariants
\begin{equation}\label{BPS}
n_d^{i,j} \in \BZ
\end{equation}
as the dimensions of the weight spaces of this $\mathfrak{sl}_2\times \mathfrak{sl}_2$-action.\footnote{For convenience, our $(i,j)$-index is different with the one used in \cite{KKP} by a constant; in particular, our invariants satisfy that $n_d^{i,j}=0$ if $i<0$ or $j<0$.} These invariants are known as the refined BPS invariants of $X$, which are expected to refine curve counting invariants for $X$ defined via Gromov--Witten/Donaldson--Thomas/Pandharipande--Thomas theory \cite{13/2}.

This paper concerns two different mathematical theories of calculating the invariants (\ref{BPS}) which connect two types of geometries. The first approach is to use the \emph{perverse filtration}, as proposed by Hosono--Saito--Takahashi \cite{HST}, Kiem--Li \cite{KL}, and Maulik--Toda \cite{MT}. More precisely, we choose $\chi \in \BZ$ coprime to $d$, and consider Le Potier's moduli space $M_{d,\chi}$ of $1$-dimensional stable sheaves $\CF$ on $\BP^2$ with 
\[
[\mathrm{supp}(\CF)] = dH \in H_2(\BP^2, \BZ), \quad \chi(\CF) = \chi;
\]
see \cite{LP1}. Here $\mathrm{supp}(-)$ denotes the Fitting support, and $H$ is the hyperplane class of $\BP^2$. It is smooth under the coprime assumption. The associated Hilbert--Chow map induces an increasing filtration 
\[
P_0H^*(M_{d,\chi}, \BQ) \subset P_1H^*(M_{d,\chi}, \BQ) \subset \cdots \subset H^*(M_{d,\chi}, \BQ),
\]
called the perverse filtration; we refer to Section \ref{Sec1.1} for a brief review. The mathematical definition of the invariant (\ref{BPS}) is the dimension of the graded piece of the perverse filtration,
\begin{equation}\label{BPS1}
n_d^{i,j} :=  \dim \mathrm{Gr}^P_i H^{i+j}(M_{d,\chi}, \BQ), \quad\!
\mathrm{Gr}_i^P = P_i/P_{i-1}.
\end{equation}
By \cite{MS_GT}, each $n_d^{i,j}$ is independent on the choice of $\chi$ as conjectured in \cite{Toda}. This realizes the original physics proposal of Gopakumar--Vafa; the moduli space $M_{d,\chi}$ is considered to be the $D$-brane moduli space, and the $\mathfrak{sl}_2\times \mathfrak{sl}_2$-action is given by Hard Lefschetz actions on 
\[
\bigoplus_{i,j} \mathrm{Gr}^P_i H^{i+j}(M_{d,\chi}, \BQ)
\]
where this vector space has the same dimension as $H^*(M_{d,\chi}, \BQ)$. However, perverse filtrations are usually mysterious and complicated, which makes the invariants (\ref{BPS1}) hard to compute.

The second proposal is due to Nekrasov--Okounkov \cite{NO}, motivated by the index in M-theory. It predicts that the refined BPS invariants should alternatively be given by Nekrasov partition functions. To be more precise, \emph{conjecturally} (\ref{BPS1}) are calculated by the \emph{equivariant index} of certain Nakajima quiver varieties. This allows us to obtain a combinatorial algorithm to find double-indexed invariants $\widetilde{n}^{i,j}_d$, which are conjectured to recover $n_{d}^{i,j}$ defined via the perverse filtration. A detailed description is given in Section \ref{Section 3}. Using this algorithm, we find that the refined BPS invariant $n_d^{i,j}$, or rather its combinatorial counterpart $\widetilde{n}_d^{i,j}$, can be expressed as a nice product formula asymptotically, which we desribe as follows.

\smallskip

Consider the generating series of two variables:
\[
F_{d, \mathrm{BPS}}(q,t): = \sum_{i,j} n_d^{i,j} q^it^j.
\]
We also set
\[
H(q,t) : =  { \mathsf{S}}^\bullet \left(
\frac{q^2+t^2+q^2 t^2}{1-qt}
\right)= 
\prod_{i \geq 0} \frac{1}{(1-(qt)^i q^2) (1-(qt)^i q^2t^2) (1-(qt)^i t^2)}
\]
where $\mathsf{S}^\bullet$ denotes the plethystic exponential. 

\begin{conj}\label{conj1}
For $i+j \leq 2d-4$, we have 
\[
n_d^{i,j} = [H(q,t)]^{i,j}.
\]
Here $[-]^{i,j}$ denotes the $q^it^j$-coefficient in the expansion $q,t \to 0$. 
\end{conj}

In particular, Conjecture \ref{conj1} predicts that each invariant $n_d^{i,j}$ stabilizes when $d \to +\infty$, and 
\[
F_{d, \mathrm{BPS}}(q,t) = H(q,t), \quad \mathrm{for}~~ d \to + \infty.
\]
As we will see in Section \ref{Section 4}, the bound $2d-4$ is expected to be optimal in Conjecture \ref{conj1} for any $d\geq 3$. 

\subsection{Tautological classes and $P=C$}\label{Sec0.2}
We propose a \emph{cohomological lift} of Conjecture \ref{conj1} using tautological classes and the Chern filtration. This is inspired by the $P=W$ conjecture for Hitchin systems as we will discuss in Section \ref{Sec0.4}.

The second and the third authors introduced in \cite{PS} the (normalized) tautological classes
\[
c_k(j) \in H^{2(k+j-1)}(M_{d,\chi}, \BQ);
\]
they are given by a \emph{normalization} of the integration over $H^j\in H^{2j}(\BP^2, \BQ)$ of the Chern character $\mathrm{ch}_{k+1}(\BF)$ associated with a universal family $\BF$; see Section \ref{sec1.2}. The main theorem of \cite{PS} is the the following generation result, where the bound $2d-4$ of Conjecture \ref{conj1} appeared naturally.

\begin{thm}[\cite{PS}]\label{thm0.2}
For $d\geq 3$, the tautological classes whose cohomological degrees $\leq 2d-4$:
\begin{equation}\label{generator}
 c_0(2), c_2(0) \in H^2(M_{d,\chi}, \BQ), \quad  c_{k}(0), c_{k-1}(1), c_{k-2}(2) \in H^{2k-2}(M_{d,\chi}, \BQ),~~~ k\in \{3,\dots, d-1 \}
\end{equation}
have no relations in $H^{*\leq 2d-4}(M_{d,\chi}, \BQ)$, and they generate $H^*(M_{d,\chi}, \BQ)$ as a $\BQ$-algebra.
\end{thm}

Therefore, we call 
\[
H^{*\leq 2d-4}(M_{d,\chi}, \BQ) \subset H^*(M_{d,\chi}, \BQ)
\]
the \emph{free part} of the cohomology. The \emph{Chern filtration} of the free part is an increasing filtration
\[
C_0H^{*\leq 2d-4}(M_{d,\chi}, \BQ) \subset C_1H^{*\leq 2d-4}(M_{d,\chi}, \BQ) \subset \cdots \subset H^{*\leq 2d-4}(M_{d,\chi}, \BQ), 
\]
where the $k$-th piece $C_kH^{*\leq 2d-4}(M_{d,\chi}, \BQ)$ is defined as the span of all monomials 
\begin{equation}\label{Chern}
\prod_{i=1}^s c_{k_i}(j_i) \in H^{* \leq 2d-4}(M_{d,\chi}, \BQ),  \quad \textrm{with }  \sum_{i=1}^s k_i \leq k.
\end{equation}

The following conjecture connects the perverse and the Chern filtrations for the free part. 

\begin{conj}[$P=C$]\label{P=C}
For $d\geq 3$, we have
\[
P_k H^{* \leq 2d-4}(M_{d,\chi}, \BQ) = C_k H^{*\leq 2d-4}(M_{d,\chi}, \BQ).
\]
\end{conj}

The next proposition follows from a direct calculation of the dimension of the Chern filtration using the freeness result of Theorem \ref{thm0.2}. In particular, Conjecture \ref{P=C} is a cohomological enhancement of Conjecture \ref{conj1}, which explains the product formula $H(q,t)$.

\begin{prop}\label{prop0.5}
Conjecture \ref{P=C} implies Conjecture \ref{conj1}.
\end{prop}

\begin{rmk}\label{rmk0.4}
The bound $2d-4$ of Conjecture \ref{P=C} is expected to be optimal. From the numerical perspective, this is due to the fact that the same bound in Conjecture \ref{conj1} is optimal for all the cases we have calculated. Moreover, from the cohomological perspective, we verify the optimality of the bound when $d=3,4$; see Remarks \ref{rmk1.4} and \ref{rmk2.11}.
\end{rmk}

The following diagram summarizes the picture above:
\begin{equation*}
    \begin{tikzcd}
     \textup{Perverse~filtration}  \arrow[d, snake it, "\,\dim \mathrm{Gr}(-)"] \arrow[rr, dashed, "\textup{free~part}"] & & \textup{Chern~filtration} \arrow[d, snake it, "\,\dim \mathrm{Gr}(-)"] \\
      F_{d,\mathrm{BPS}}(q,t) \arrow[rr,dashed, "{i+j \leq 2d-4}"] & & H(q,t).
\end{tikzcd}
\end{equation*}

\subsection{Low degree cases}

When $d=1,2$, the moduli space $M_{d,\chi}$ is the projective space $\BP H^0(\BP^2, \CO_{\BP^2}(d))$. The perverse filtration is trivial:
\[
\mathrm{Gr}^P_kH^m(M_{d,\chi}, \BQ) = 0,\quad \mathrm{for}\ \; k <m.
\]
From now on, we focus on the non-trivial cases $d\geq 3$. The main result of this paper is the following.

\begin{thm}\label{Thm0.6}
Conjecture \ref{P=C} holds for degrees $d =3, 4$.
\end{thm}

More precisely, we describe the cohomology ring $H^*(M_{d,\chi}, \BQ)$ for $d\leq 4$ in terms of the generators (\ref{generator}); then the perverse filtration can be calculated using the ring structure. As a byproduct, we compute the invariants (\ref{BPS1}) explicitly for $d=3 ,4$.

\begin{thm}[\textit{c.f.} Conjecture \ref{conj0}]\label{thm0.7}
For degrees $d=3,4$, the invariants (\ref{BPS1}) defined by the perverse filtration are matched with the refined BPS invariants defined by the Nekrasov partition functions or the refined Pandharipande--Thomas invariants. 
\end{thm}

In particular, our calculation shows that the (unrefined) BPS invariants $n_{g,d}$ induced by the perverse filtration \cite{HST, KL, MT} match with those defined by Gromov--Witten or Pandharipande--Thomas theory; this yields the Gopakumar--Vafa/Gromov--Witten correspondence for $d=3,4$.

\subsection{$\bf P=W$ and $\bf P=C$}\label{Sec0.4}

The $P=C$ phenomenon is closely related to the $P=W$ conjecture of de Cataldo, Hausel, and Migliorini \cite{dCHM1}.

Let $\Sigma$ be a smooth projective curve of genus $g\geq 2$, and let $n,d$ be two coprime integers. The moduli space $M_{\mathrm{Dol}}$ of stable Higgs bundles on $\Sigma$ of rank $n$ and degree $d$ admits a perverse filtration induced by the associated Hitchin system. The $P=W$ conjecture asserts that the perverse filtration of $M_{\mathrm{Dol}}$ is matched with the double-indexed weight filtration associated with the corresponding character variety $M_{\mathrm{B}}$ via the non-abelian Hodge correspondence:
\[
``P=W": \quad P_k H^*(M_{\mathrm{Dol}}, \BQ) = W_{2k} H^*(M_{\mathrm{B}}, \BQ).
\]
This conjecture has been proven recently in \cite{MS_PW} and \cite{HMMS} independently.

Analogous to the tautological classes $c_k(j)$, we consider the tautological classes for the Hitchin moduli space
\[
c_k(\gamma) \in H^*(M_{\mathrm{Dol}}, \BQ), \quad k\in \BZ_{\geq 0}, ~ \gamma \in H^*(\Sigma,\BQ)
\]
given by a normalization of the integration over $\gamma$ of $\mathrm{ch}_k(\CU)$ \cite[Section 0.3]{dCMS}, where $\CU$ is a fixed universal bundle. These classes are proven to generate the cohomology \cite{Markman} and their weights on $M_\mathrm{B}$ were calculated in \cite{Shende}. Consequently, the $P=W$ conjecture is equivalent to:
\begin{equation}\label{P=C_Higgs}
``P=C": \quad P_kH^*(M_{\mathrm{Dol}}, \BQ) = C_kH^*(M_{\mathrm{Dol}}, \BQ)
\end{equation}
where the Chern filtration is defined by the Chern degrees of the tautological classes as in (\ref{Chern}); see \cite[Conjecture 0.3]{dCMS}. In fact, all the approaches in \cite{dCHM1, dCMS, MS_PW, HMMS} for (certain cases of) $P=W$ are to prove $P=C$ via various techniques. 

In view of (\ref{P=C_Higgs}), Conjecture \ref{P=C} is an analog of $P=W$ for Hitchin systems. We note that the major difference between Conjecture \ref{P=C} and (\ref{P=C_Higgs}) is that the former only holds for the free part as explained by Remark \ref{rmk0.4}. This may be due to the fact that the fibration associated with $M_{d,\chi}$ fails to be Lagrangian.\footnote{For another type of Lagrangian fibration --- the Beauville--Mukai system associated with a K3 or an abelian surface, there is a version of $P=C$ for the total cohomology; see \cite[Theorem 2.1]{dCMS}.} 

The enumerative geometry perspective of $P=W$ concerning the refined BPS invariants for the local curve $T^*\Sigma \times \BC$ was discussed in \cite{CDP}. It connects the conjecture of Hausel--Rodriguez-Villegas \cite{HRV} on mixed Hodge polynomials of character varieties with certain equivariant index of the Hilbert scheme $\mathrm{Hilb}(\BC^2,n)$ of points on $\BC^2$.

\subsection{Relations to other work}
Recently there has been much work in connections between moduli of 1-dimensional sheaves on $\BP^2$ and enumerative geometry for local or logarithmic $\BP^2$ \cite{PB,PB2,BFGW}. In the case of a K3 surface or an abelian surface $S$, the Gopakumar--Vafa theory for the Calabi--Yau 3-fold $S\times \BC$ is closely related to compact hyper-K\"ahler geometries; the $P=C$ phenomenon was deduced in \cite{dCMS} and the invariants (\ref{BPS1}) were calculated in terms of Hodge numbers of certain compact hyper-K\"ahler manifolds \cite{SY, HLSY}; this is matched with the prediction from physics \cite{KKV, KKP}. Refined BPS invariants for local $\BP^2$ have been studied via stable pairs \cite{CKK}.

The ``$P=C$" phenomenon also appeared in geometric representation theory for certain affine Springer fibers \cite{OY1,OY2}.

\subsection{Acknowledgements}
We wish to thank Pierrick Bousseau, Mark Andrea de Cataldo, Davesh Maulik, Nikita Nekrasov, Andrei Okounkov, and Rahul Pandharipande for many conversations over the years on BPS invariants, DT/PT invariants, and perverse filtrations. J.S. was supported by the NSF grant DMS-2134315.

\section{Perverse filtrations, moduli spaces, and tautological classes}\label{Sec1}

In this section, we review some basic facts about perverse filtrations, moduli of 1-dimensional stable sheaves on $\BP^2$, and the (normalized) tautological classes introduced in \cite{PS}. In Proposition \ref{prop1.2}, we reinterpret the normalization of \cite{PS} as the only one enforcing $``P=C"$ to hold for $H^{*\leq 2}$. As a toy example of the calculation in the next section, we conclude this section with the proofs of Theorem \ref{Thm0.6} and Theorem \ref{thm0.7} for $d=3$.


\subsection{Perverse filtrations}\label{Sec1.1}

Let $f: X \to Y$ be a proper morphism between irreducible nonsingular quasiprojective varieties with $\dim X = a$ and $\dim Y = b$. Let $r$ be the \textit{defect of semismallness} of $f$:
\[
r: = \dim X \times_Y X - \dim X.
\]
For convenience, we further assume that $f$ has equidimensional fibers, so that $r = a-b$. The perverse filtration 
\[
P_0H^m(X, \BQ) \subset P_1H^m(X, \BQ) \subset \dots \subset P_{2r}H^m(X, \BQ) =  H^m(X, \BQ)
\]
is an increasing filtration on the cohomology of $X$ governed by the topology of the morphism $f$; it is defined to be
\[
P_iH^m(X, \BQ) := \mathrm{Im}\left\{ H^{m-b}(Y, {^\mathbf{p}\tau_{\leq i}} (Rf_* \BQ_X[b])) \to H^m(X, \BQ)\right\}
\]
where $^\mathbf{p}\tau_{\leq * }$ is the perverse truncation functor \cite{BBD}. We say that a class $\gamma \in H^*(X, \BQ)$ has \emph{perversity} $k$, if 
\[
\gamma \in P_kH^*(X, \BQ) \setminus P_{k-1}H^*(X, \BQ).
\]

In general, the perverse filtration associated with a morphism is very complicated and hard to compute, as it relies on the mysterious perverse truncation functor. In the case when the target $Y$ is projective, we may describe the perverse filtration via an ample class on $Y$ as follows.

We fix $\eta$ to be an ample class on $Y$. Its pullback gives a class $\xi:= f^*\eta \in H^2(X, \BQ)$, which acts on the rational cohomology of $X$ via cup product:
\[
\xi: H^m(X, \BQ) \xrightarrow{~~-\cup \xi~~} H^{m+2}(X, \BQ).
\]

\begin{prop}[\emph{c.f.} \cite{dCM0} Proposition 5.2.4]\label{dCM0}
With the notation as above, we have
    \begin{equation}\label{prop1.1}
    P_k H^m(X,\mathbb{Q})=\sum_{i\geq 1} \left( \mathrm{Ker} (\xi^{b+k+i-m}) \cap \operatorname{Im} (\xi^{i-1})\right) \cap H^m(X, \mathbb{Q}).
    \end{equation}
\end{prop}

As in (\ref{BPS1}), we are interested in the dimension of the graded piece of the perverse filtration 
\[
\dim \mathrm{Gr}^P_iH^{i+j}(X, \BQ)
\]
which can be expressed via the decomposition theorem \cite{BBD}. More precisely, applying the decomposition theorem to $f: X\to Y$, we obtain that
\begin{equation*}
R\pi_*\BQ_X[b]\simeq \bigoplus_{i=0}^{2r}\mathcal{P}_i[-i] \in D^b_c(Y)
\end{equation*}
with $\mathcal{P}_i$ a semisimple perverse sheaf on $Y$. The perverse filtration can be identified as
\[
P_kH^m(X,\BQ)=\mathrm{Im}\Big\{H^{m-b}(Y, \bigoplus_{i=0}^k\mathcal{P}_i[-i])\to H^m(X,\BQ)\Big\},
\]
and hence
\begin{equation}\label{perv_coh}
\dim \mathrm{Gr}_i^P H^{i+j} (X, \BQ) =  \dim H^{j-b}(Y, \CP_i).
\end{equation}

\subsection{Tautological classes for moduli spaces}
\label{sec1.2}
From now on we focus on the moduli of 1-dimensional stable sheaves. We review the tautological classes introduced in \cite[Section 1.1]{PS}. As we will show in Proposition \ref{prop1.2}, the normalization we used in \cite{PS} is crucial for the $P=C$ conjecture to hold.

Recall that the moduli space $M_{d,\chi}$ parameterizes 1-dimensional stable sheaves $\CF$ on $\BP^2$ with 
\[
[\mathrm{supp}(\CF)] =dH, \quad \chi(\CF) = \chi.
\]
Here the stability condition is with respect to the slope
\[
\mu( \CE) = \frac{\chi(\CE)}{c_1(\CE)\cdot H} \in \BQ.
\]
It admits a Hilbert--Chow map
\[
h: M_{d,\chi} \to \BP H^0(\BP^2, \CO_{\BP^2}(d)), \quad \CF \mapsto \mathrm{supp}(\CF),
\]
sending a sheaf to its Fitting support. This is a flat and proper map, which induces a perverse filtration 
\[
P_0H^*(M_{d,\chi}, \BQ) \subset P_1H^*(M_{d,\chi}, \BQ) \subset \cdots \subset H^*(M_{d,\chi}, \BQ)
\]
by the discussion of Section \ref{Sec1.1}. The assumption $\mathrm{gcd}(d,\chi) =1$ garantees that the stability and the semistability conditions coincide. Its connection to the enumerative geometry of the local Calabi--Yau 3-fold $X = \mathrm{Tot}(K_{\BP^2})$ relies on the fact that $M_{d,\chi}$ can also be viewed as the moduli of 1-dimenensional stable sheaves on $X$ with the same numerical data. 

Let $\BF$ be a universal family over $\BP^2 \times M_{d,\chi}$. For a stable sheaf $[\CF] \in M_{d,\chi}$, the restriction of $\BF$ to the fiber $\BP^2 \times [\CF]$ recovers $\CF$. Since the choice of $\BF$ is not unique, we need to normalize its Chern character $\mathrm{ch}(\CF)$ to obtain cohomolgy classes $c_k(j)$ of Section \ref{Sec0.2} which are independent on $\BF$. We review the construction as follows.

\smallskip

For a universal family $\BF$ and a class\footnote{We change the notation of the normalizing class in \cite{PS} to $\delta$ as $\alpha$ is used in Theorem \ref{thm2.1} below.}
\begin{equation*}\label{alpha}
\delta= \pi_P^* \delta_P + \pi_M^* \delta_M \in H^2(\BP^2\times M_{d,\chi}, \BQ), \quad \textup{with}~~ \delta_P \in H^2(\BP^2, \BQ), ~~~\delta_M\in H^2(M_{d,\chi}, \BQ),
\end{equation*}
we consider the twisted Chern character
\[
\mathrm{ch}^\delta(\BF) := \mathrm{ch}(\BF) \cdot \mathrm{exp}(\delta),
\]
and we denote $\mathrm{ch}_k^\delta(\BF)$ its degree $k$-part. For $H^j \in H^{2j}(\BP^2, \BQ)$, we set
\[
c^\delta_k(j): = \int_{H^j} \mathrm{ch}_{k+1}^\delta(\BF) =\pi_{M*}\left( \pi_P^* {H^j} \cdot \mathrm{ch}_{k+1}^\delta(\BF)\right) \in H^{2(k+j-1)}(M_{d,\chi}, \BQ),
\]
where $\pi_P$ and $\pi_M$ are the natural projections from the product to $\BP^2$ and $M_{d,\chi}$, respectively.

\smallskip

The next proposition asserts that $P=C$ for $H^{*\leq 2}(M_{d,\chi}, \BQ)$ holds for the $\delta$-twisted tautological classes $c^\delta_k(j)$, if and only if  $\delta$ is the normalization class chosen in \cite{PS}.

\begin{prop} \label{prop1.2}
With the above notation, we have:
\begin{enumerate}
    \item[(i)] There exists a unique $\delta_0$ satisfying the condition that the classes \[
    c_k^{\delta_0}(j) \in H^{2(k+j-1)}(M_{d,\chi}, \BQ) \quad \mathrm{with}~~ k+j \leq 2
    \]
    has perversity $k$.
 \item [(ii)] We define $c_k(j)$ using the class $\delta_0$ determined by (i). Then 
 \[
 P_kH^{*\leq 2}(M_{d,\chi}, \BQ) = C_kH^{*\leq 2}(M_{d,\chi}, \BQ).
 \]
\end{enumerate}
\end{prop}

\begin{proof}
For (i), since  $H^0(M_{d,\chi}, \BQ) = P_0H^0(M_{d,\chi}, \BQ)$, we require that
\[
c^{\delta_0}_1(0) = 0 \in H^0(M_{d,\chi}, \BQ).
\]
Moreover, since $H^2(M_{d,\chi}, \BQ)$ is 2-dimensional generated by an ample class on the base pulled back via $h$ and an $h$-relative ample class, we have
\[
P_1H^2(M_{d,\chi}, \BQ) \setminus P_0H^2(M_{d,\chi}, \BQ) = 0.
\]
Therefore, we also require that 
\[
c^{\delta_0}_1(1) = 0 \in H^2(M_{d,\chi}, \BQ).
\]
These conditions determine $\delta_0$ uniquely by \cite[Proposition 1.2]{PS}. In particular the choice of $\delta_0$ using (i) recovers the tautological classes $c_k(j)$ introduced in \cite{PS}.

Part (ii) follows from \cite[Proposition 1.3(c)]{PS}. More precisely, since $c_0(2)$ is pulled back from the base $\BP H^0(\BP^2, \CO_{\BP^2}(d))$ and $c_2(0)$ is relative ample, we have
\[
c_0(2) \in P_0H^2(M_{d,\chi}, \BQ),~~c_2(0) \in P_2 H^2(M_{d,\chi}, \BQ) \setminus P_1H^2(M_{d,\chi}, \BQ).
\]
This proves $P=C$ for $H^{*\leq 2}(M_{d,\chi}, \BQ)$.
\end{proof}

As a consequence of Proposition \ref{dCM0}, we may determine the perverse filtration for $M_{d,\chi}$ using the tautological class $c_0(2) \in H^2(M_{d,\chi}, \BQ)$; this is our main tool of calculating the perverse filtration for low degrees $d$.

\begin{cor}\label{cor1.3}
The perverse filtration $P_\bullet H^*(M_{d,\chi}, \BQ)$ is characterized by the formula (\ref{prop1.1}) with $\xi = c_0(2)$. 
\end{cor}

\subsection{Symmetries}\label{Sec1.3}
We have two types of symmetries between the moduli spaces $M_{d,\chi}$:
\begin{enumerate}
    \item[(i)] The first type is given by the isomorphism
    \[
    \phi_1: M_{d,\chi} \xrightarrow{\sim} M_{d,\chi+d}, \quad \CF \mapsto \CF\otimes \CO_{\BP^2}(1).
    \]
    \item[(ii)] The second type is given by the isomorphism
    \[
    \phi_2: M_{d,\chi} \xrightarrow{\sim} M_{d, -\chi}, \quad \CF \mapsto \CE \kern -1.5 pt \mathit{xt}^1(\CF, \omega_{\BP^2}).
    \]
\end{enumerate}
Both symmetries preserve the morphism $h: M_{d,\chi} \to \BP H^0(\BP^2, \CO_{\BP^2}(d))$.
 Furthermore, by \cite[Proposition 1.4]{PS}, the tautological classes $c_k(j)$ are preserved (up to a sign) by the symmetries (i) and (ii) above.

Thus, in order to prove Theorem \ref{Thm0.6} and Theorem \ref{thm0.7} for $M_{d,\chi}$, it suffices to establish them for $M_{d,\chi'}$ with some $\chi'$ satisfying that 
\[
\chi' = \pm \chi ~~~ \mathrm{mod}~d.
\]

\subsection{Degree 3 case}

We conclude Section \ref{Sec1} with a complete calculation of the perverse filtration and their dimensions in the degree 3 case; in particular we prove Theorem \ref{Thm0.6} and Theorem \ref{thm0.7} for $d=3$.

\begin{proof}[Proof of Theorem \ref{Thm0.6} for $d=3$]
In this case the bound $2d-4 = 2$; therefore Conjecture \ref{P=C} only concerns $H^{* \leq 2}$, which follows immediately from Proposition \ref{prop1.2}(ii). \qedhere
\end{proof}

\begin{rmk}\label{rmk1.4}
    Since $h: M_{3,\chi} \to \BP^9$ is an elliptic fibration, we know that 
    \[
 P_2H^*(M_{3,\chi}, \BQ) = H^*(M_{d,\chi}, \BQ).
    \]
    Therefore by considering 
    \[
c_2(0)^2 \in H^4(M_{3,\chi}, \BQ),
    \]
    it is obvious that $P=C$ breaks down for $H^4(M_{3,\chi}, \BQ)$. In particular, the bound $2d-4$ of Conjecture \ref{P=C} is optimal for $d=3$. 
\end{rmk}

\begin{proof}[Proof of Theorem \ref{thm0.7} for $d=3$]

We need to show that the invariants
\[
n_3^{i,j} = \dim \mathrm{Gr}^P_iH^{i+j}(M_{3,\chi}, \BQ)
\]
are matched with the coefficients of $\widetilde{F}_{3, \mathrm{BPS}}(q,t)$ obtained in Section \ref{Section 4}. This can be achieved by calculating the perverse filtration using the ring structure \cite[Section 1.3]{PS} combined with Corollary \ref{cor1.3}. We leave this as an exercise to the reader as we will use this method to treat the $d=4$ case in Section \ref{Sec2} where the calculation is much more complicated. 

Here we give another proof via the decomposition theorem (\ref{perv_coh}). Since $h: M_{3,\chi} \to \BP^9$ is an elliptic fibration, we have by relative Hard Lefschetz the decomposition theorem associated with $h$:
\begin{equation}\label{decomp}
Rh_* \BQ[9] \simeq (\BQ[9]) \oplus \CP_1[-1] \oplus  (\BQ[9])[-2].
\end{equation}
Therefore, the only unknown is the invariant 
\[
n_3^{1,j} = \dim H^{j-9}(\BP^9, \CP_1), \quad j \in \BZ.
\]
This can be calculated by taking $H^j(-, \BQ)$ in (\ref{decomp}):
\[
\dim H^j(\BP^9, \BQ) + n_3^{1,j} + \dim H^{j-2}(\BP^9, \BQ) = \dim H^j(M_{3,\chi}, \BQ) = \dim  H^j(\BP^2 \times \BP^8, \BQ),
\]
where the last equation follows from the fact that $M_{3,\chi}$ is a projective bundle over $\BP^2$, see \cite{LP1}. Therefore we have calculated all the refined BPS invariants $n_3^{i,j}$, which completes the proof by comparing with $\widetilde{F}_{3, \mathrm{BPS}}(q,t)$ in Section \ref{Section 4}.
\end{proof}

\section{\texorpdfstring{$P=C$}{P=C} for degree 4 and matching BPS invariants} \label{Sec2}

\subsection{Overview}
We complete the proof of Theorem \ref{Thm0.6} and Theorem \ref{thm0.7} in this section. Since we concern the case $d=4$, in view of the discussion of Section \ref{Sec1.3}, we only need to prove both theorems for the moduli space $M_{4,1}$. From now on, we only consider the case $d=4, \chi =1$.

The cohomology ring $H^*(M_{4,1}, \BQ)$ has been calculated by Chung--Moon \cite{CM} explicitly in terms of generators given by certain geometric classes. In order to prove $P=C$, we apply Chung--Moon's result to calculate the ring structure of $H^*(M_{4,1}, \BQ)$ using the tautological classes $c_k(j)$. Then Corollary \ref{cor1.3} allows us to write the perverse filtration in terms of the tautological classes.

Our main technical theorem of this section is Theorem \ref{prop2.9} which provides the translation between the geometric classes $\alpha, \beta, x,y,z$ in Theorem \ref{thm2.1} below and the tautological classes of \cite{PS}.


\subsection{Cohomology of $M_{4,1}$}

We first recall the following theorem due to Chung--Moon \cite{CM}. For an algebraic class in $H^{2i}(M_{4,1}, \BQ)$, we say that this class is of algebraic degree $i$.

\begin{thm}[{\cite[Theorem 6.5]{CM}}]
\label{thm2.1}
The Chow ring of $M_{4,1}$ is given by\footnote{We correct a typo in the original paper.}
\begin{align*}
  A^*(M_{4,1})\simeq &\; \mathbb{Q}[\alpha, \beta, x,y,z]/\langle xz-yz, \beta^2z-3yz-9z^2, 3\alpha^2z-\alpha \beta z+yz,\beta^2y-3y^2-9yz,\\& \beta^2 x -xy-3y^2-3\alpha \beta z -9yz +9z^2, \beta^4 +3x^2 -9xy-3y^2-54yz-81z^2,\\& \beta yz +9\alpha z^2 -3\beta z^2, 2\beta xy-3\beta y^2 -9\alpha yz -27\alpha z^2 +9\beta z^2, 3\beta x^2 -7 \beta y^2 -36 \alpha yz\\& -108 \alpha z^2 +36 \beta z^2, \alpha^{12}+3\alpha^{11}\beta +3 \alpha^{10}(\beta^2 +2x -y)+ \alpha^9 (-\beta^3 +12 \beta x + 2 \beta y)\\& +3\alpha^8(9x^2 -16 xy +17 y^2) + 28 \alpha^7 \beta y^2 + 56 \alpha^6 y^3 + 201 \alpha \beta z^5 -19 yz^5 -613 z^6,\\& 6\alpha^{10} xy - 12 \alpha^{10} y^2 -10 \alpha^9 \beta y^2 -45 \alpha^8 y^3 -104 \alpha \beta z^6 + 2yz^6+310 z^7\rangle,
\end{align*}
where $\alpha, \beta$ are of algebraic degree $1$ and $x,y,z$ of degree $2$. This also gives the cohomology ring $H^{2*}(M_{4,1}, \mathbb{Q})$, with the degrees of the generators doubled.

\end{thm}

The class $\alpha$ can be described as the locus of $\mathcal{F}\in M_{4,1}$ such that a fixed point $p\in \BP^2$ lies in $\mathrm{supp}(\mathcal{F})$; see \cite[Proposition 7.8]{CM}. Otherwise said, it is the pull-back of a hyperplane class on $\mathbb{P}^{14} = \BP H^0(\BP^2, \CO_{\BP^2}(4))$ via $h: M_{4,1} \to \BP^{14}$. For geometric descriptions of the other generators, see Sections \ref{Sec2.1.1} and \ref{Sec2.3.2}, and also \cite[Section 7]{CM}. As we discuss in Section \ref{Sec2.3}, only the descriptions of $\beta$ and $z$ will be needed for our calculations.

\subsection{Comparing generators.}\label{Sec2.3}

The major part of this section consists of comparing the generators in Theorem \ref{thm2.1} with the five tautological generators
\begin{equation}\label{5generators}
c_0(2), c_2(0), c_1(2), c_2(1), c_3(0) \in H^{*}(M_{4,1}, \mathbb{Q})
\end{equation}
given in \cite{PS}.
\smallskip

To begin with, note that
\[
\alpha=c_0(2)
\]
by the discussion above, both being the pull-back of a hyperplane class on the base $\mathbb{P}^{14}$. Thus we write the two classes interchangeablely in what follows.

\medskip

To determine the remaining classes of Theorem \ref{thm2.1} in terms of (\ref{5generators}), we proceed by the following three steps:

\begin{enumerate}
    \item[(i)] Compute the classes $\beta$ and $z$ explicitly in terms of (\ref{5generators}); this is carried out in Sections \ref{Sec2.1.1} and \ref{Sec2.3.2}.
    \smallskip
    
    \item[(ii)] Compute the total Chern class $c(\mathcal{T}_M)$ of $M_{4,1}$ in terms of (\ref{5generators}); this is carried out in Section \ref{Sec2.3.3}.
    \smallskip
    
    \item[(iii)] Comparing the result of (ii) with \cite[Proposition 7.5]{CM}, we obtain two identities by taking the terms in $c(\mathcal{T}_M)$ of algebraic degrees $2$ and $3$. This allows us to find the expressions for $x$ and $y$ in terms of (\ref{5generators}).    
\end{enumerate}

\begin{rmk}
Most steps in (i,\,ii,\,iii) are technical calculations in intersection theory. For the first time reading this article, the reader may skip this part and jump to Theorem \ref{prop2.9} directly.
\end{rmk}

\subsubsection{The class $\beta$.}\label{Sec2.1.1} We first introduce some notation and recall Proposition \ref{prop2.2} from \cite{CM} which will be needed.

A general element $\mathcal{F}\in M_{4,1}$ has a unique nonzero section $s: \mathcal{O}_{\mathbb{P}^2} \to \mathcal{F}$ up to scalar multiplication, whose cokernel $Q_\mathcal{F}$ has finite support. We denote by $L$ the closure of the locus of $\mathcal{F}\in M_{4,1}$ such that $Q_\mathcal{F}$ meets a fixed line. Let $O$ be the closure of the locus of $\mathcal{F}\in M_{4,1}$ such that $Q_\mathcal{F}$ contains a fixed point. 

\begin{prop}[{\cite[Proposition 7.11]{CM}}]\label{prop2.2}
With the notation as above, we have:
\begin{enumerate}
    \item[(i)] $L=-\beta$ in $H^{2}(M_{4,1}, \mathbb{Q})$.
    \item[(ii)] $O=x-y$ in $H^4(M_{4,1}, \BQ)$.
\end{enumerate}
\end{prop}



Let $K(-)$ be the Grothendieck group of coherent sheaves. Consider the group homomorphism $\lambda: K(\BP^2) \to \mathrm{Pic}(M_{4,1})$ given by the following composition
\[
 K(\mathbb{P}^2) \to K(\mathbb{P}^2 \times M_{4,1})\to K(M_{4,1})\to \mathrm{Pic}(M_{4,1}).
\]
Explicitly, it is defined by 
\[
\lambda(\nu)=\det (q_* (\mathbb{F}\otimes p^* \nu)) \in \mathrm{Pic}(M_{4,1}), \ \  \nu \in K(\mathbb{P}^2).
\]
Here $\mathbb{F}$ is a fixed universal sheaf on $\mathbb{P}^2\times M_{4,1}$ and $p$ (resp. $q$) is the projection to the first (resp. second) factor. We write
\[
D:= \lambda(-4 \mathcal{O}_{\mathbb{P}^2} + \mathcal{O}_H) \in H^2(M_{4,1}, \BQ)
\]
with $H \subset \BP^2$ a hyperplane.

\begin{prop}[{\cite[Proposition 2.5]{CC15}}] Under the above notation, we have
\begin{equation}
    \label{D}
D=-3\alpha +L
\end{equation}
\end{prop}


\smallskip

Combining Proposition \ref{prop2.2}(i) and (\ref{D}), we deduce that
\begin{equation}
\label{beta}
\beta=-3\alpha-D.
\end{equation}
Since $\alpha = c_0(2)$, it suffices to express $D$ in terms of $c_k(j)$. For this purpose, we first present a general lemma that expresses the normalization class $\delta_0$ of Proposition \ref{prop1.2}:


\begin{lem}\label{lemma2.5}
For a general moduli space $M_{d,\chi}$, the normalization class $\delta_0$ is given by
\begin{equation}
    \label{delta}
    \delta_0=\left(\frac{3}{2}-\frac{\chi}{d}\right)\cdot H-\frac{1}{d}\left(\left(\frac{3}{2}-\frac{\chi}{d}\right)c_0(2)+e_1(1)\right),
\end{equation}
where $e_1(1):=\int_H \mathrm{ch}_2(\mathbb{F}) \in H^2(M_{d,\chi}, \mathbb{Q})$.
\end{lem}

\begin{proof}
This is a direct calculation following the proof of \cite[Proposition 1.2]{PS}.
\end{proof}

The precise expression of the second term in (\ref{delta}) is of little significance for our purpose, and thus we write simply $\delta_0=\left(\frac{3}{2}-\frac{\chi}{d}\right)\cdot H -\gamma$ from now on. Specializing to $M_{4,1}$, we have
\[
\delta_0=\frac{5}{4}H-\gamma \in H^2(\mathbb{P}^2 \times M_{4,1}, \BQ).
\]

\smallskip

Now we can compute the class $D\in H^2(M_{4,1}, \mathbb{Q})$. By definition, we have
\begin{align*}
D=\lambda(-4 \mathcal{O}_{\mathbb{P}^2} + \mathcal{O}_H)&=\mathrm{ch}_1\big(\det (q_* (\mathbb{F}\otimes p^*(-4 \mathcal{O}_{\mathbb{P}^2}+ \mathcal{O}_H)))\big)\\&= \mathrm{ch}_1\big(q_* (\mathbb{F}\otimes p^*(-4 \mathcal{O}_{\mathbb{P}^2}+ \mathcal{O}_H))\big)
\\&= -4\cdot \mathrm{ch}_1(q_* \mathbb{F})+\mathrm{ch}_1\big(q_* (\mathbb{F}\otimes p^*\mathcal{O}_{H})\big).
\end{align*}
Using the Grothendieck--Riemann--Roch theorem and the projection formula, we obtain
\begin{align*}
\mathrm{ch}(q_*\mathbb{F})&=q_*(\mathrm{ch}(\mathbb{F})\cdot \mathrm{td}(\mathbb{P}^2))\\&= q_*\left(\mathrm{ch}^{\delta_0}(\mathbb{F}) \cdot \exp\left(-\frac{5}{4}H\right)\cdot \mathrm{td}(\mathbb{P}^2)\right)\cdot \exp(\gamma).
\end{align*}

Since $\mathrm{td}(\mathbb{P}^2)=1+\frac{3}{2}H +H^2$, it follows that
\[
\mathrm{ch}_1(q_* \mathbb{F})=-\frac{3}{32}c_0(2)+c_2(0) +\gamma.
\]
Similarly, we have
\[
\mathrm{ch}_1(q_*\mathbb{F}\otimes \mathcal{O}_H)=-\frac{1}{4} c_0(2) + 4\gamma.
\]
Therefore we arrive at
\[
D=-4\cdot \mathrm{ch}_1(q_* \mathbb{F})+\mathrm{ch}_1\big(q_* (\mathbb{F}\otimes p^*\mathcal{O}_{H})\big)=\frac{1}{8} c_0(2) -4 c_2(0).
\]

We conclude from (\ref{beta}) that
\[
\beta=-\frac{25}{8}c_0(2) + 4c_2(0).
\]

\subsubsection{The class $z$.}\label{Sec2.3.2} We compute the class $z$ by an application of the Porteous formula. Denote by $\mathcal{C}_4$ the universal quartic curve in $\mathbb{P}^2 \times \mathbb{P}^{14}$. This sits naturally in $M_{4,1}$ and can be described as the \textit{Brill--Noether locus} of sheaves $\mathcal{F}\in M_{4,1}$ with $\dim H^0(\mathbb{P}^2, \mathcal{F})=2$; see \cite{DM11}. For a sheaf not belonging to $\mathcal{C}_4$, we have $\dim H^0(\mathbb{P}^2, \mathcal{F})=1$. The following proposition gives a geometric description of the class $z$:

\begin{prop}[{\cite[Proposition 7.7]{CM}}]
\label{prop2.6} We have $z=[\mathcal{C}_4]$ in $H^{*}(M_{4,1}, \mathbb{Q})$.
\end{prop}

The Brill--Noether locus $\mathcal{C}_4$ has codimension two and can be viewed as a degeneracy locus of a map between vector bundles, as we explain now. Fix a universal sheaf $\mathbb{F}$ on $\mathbb{P}^2 \times M_{4,1}$, and consider the second projection $q: \mathbb{P}^2 \times M_{4,1} \to M_{4,1}$. The derived push-forward 
\[
{R}q_* \mathbb{F}\in D^b\mathrm{Coh}(M_{4,1})
\]
admits a two-term resolution $\phi:K^0 \to K^1$ by vector bundles, as it computes the cohomology groups on curves. For a sheaf $\mathcal{F}\in M_{4,1}$ supported on a curve $C$, we have the exact sequence
\[
0\to H^0(C, \mathcal{F})\to K^0(\mathcal{F}) \xrightarrow{\phi(\mathcal{F})} K^1(\mathcal{F}) \to H^1(C, \mathcal{F}) \to 0.
\]
Denote by $e$ the rank of the vector bundle $K^0$. We have
\[
\chi(\mathcal{F})=\dim H^0(C,\mathcal{F})-\dim H^1(C, \mathcal{F})=1,
\]
and thus $f:=\operatorname{rank} K^1(\CF)=e-1$. Recall that
\[
\mathcal{C}_4=\{\mathcal{F} \in M_{4,1}\mid \dim H^0(\mathbb{P}^2, \mathcal{F})=\dim H^0(C, \mathcal{F})=2\};
\]
we see that $\mathcal{C}_4$ coincides with the degeneracy locus $M_{e-2}(\phi)$ where the map $\phi: K^0 \to K^1$ between vector bundles has rank $\leq e-2$. Moreover, it has the expected codimension
\[
(e-(e-2))(f-(e-2))=2.
\]
Thus by the Porteous formula, we obtain that
\begin{align*}
[\mathcal{C}_4]=[M_{e-2}(\phi)]&=\Delta_1^2\left[\frac{c(K^1)}{c(K^0)}\right]\\&= c_1(K^1-K^0)^2-c_2(K^1-K^0)\\&= c_1(-q_*\mathbb{F})^2-c_2(-q_*\mathbb{F})\\&=\frac{1}{2}\mathrm{ch}_1(q_*\mathbb{F})^2-\mathrm{ch}_2(q_*\mathbb{F}).
\end{align*}

\smallskip

On the other hand, we have computed $\mathrm{ch}(q_*\mathbb{F})$ in Section \ref{Sec2.1.1}:
\[
\mathrm{ch}(q_* \mathbb{F})=1-\frac{3}{32}c_0(2) + c_2(0) +\gamma -\frac{3}{32}c_1(2) +\frac{1}{4} c_2(1) + c_3(0) +\gamma \cdot \big(-\frac{3}{32}c_0(2) +c_2(0)\big)+\frac{\gamma^2}{2}+ \cdots,
\]
where the omitted terms have algebraic degrees $\geq 3$. We conclude by Proposition \ref{prop2.6} that
\[
z=-c_3(0)-\frac{1}{4}c_2(1)+\frac{3}{32}c_1(2)+\frac{1}{2}(c_2(0)-\frac{3}{32}c_0(2))^2.
\]

\subsubsection{The total Chern class.} \label{Sec2.3.3} The goal of this subsection is to compute the total Chern class $c(\mathcal{T}_{M})$ in terms of (\ref{5generators}). As before, we fix a universal sheaf $\mathbb{F}$ on $\mathbb{P}^2 \times M_{4,1}$. Since $M_{4,1}$ is a smooth projective variety, the tangent space at a sheaf $\mathcal{F} \in M_{4,1}$ is given by
\[
\mathcal{T}_\mathcal{F}=\mathrm{Ext}^1(\mathcal{F}, \mathcal{F}).
\]

Consider the object ${R}\mathcal{H} \kern -1.2pt \mathit{om}(\mathbb{F}, \mathbb{F}) \in D^b\mathrm{Coh}(\mathbb{P}^2\times M_{4,1})$. The derived push-forward
\[
{R}q_*{R}\mathcal{H} \kern -1.2pt \mathit{om}(\mathbb{F}, \mathbb{F}) \in D^b \mathrm{Coh}(\mathbb{P}^2 \times M_{4,1})
\]
admits a three-term resolution $L^0\to L^1 \to L^2$ by vector bundles. For a sheaf $\mathcal{F}\in M_{4,1}$, the $i$-th cohomology of the sequence
\begin{equation}
L^\bullet(\mathcal{F}):\; 0\to L^0(\mathcal{F}) \to L^1(\mathcal{F}) \to L^2(\mathcal{F})\to 0
\end{equation}
computes the extension group $\mathrm{Ext}^i(\mathcal{F}, \mathcal{F})$. We have
\[
H^0(L^\bullet(\mathcal{F}))=\mathrm{Hom}(\mathcal{F}, \mathcal{F})\simeq \mathbb{C}
\]
by the stability of $\mathcal{F}$. The first cohomology is 
\[
H^1(L^\bullet(\mathcal{F}))=\mathrm{Ext}^1(\mathcal{F}, \mathcal{F})=\mathcal{T}_\mathcal{F}.\] 
For the second cohomology, we have by Serre duality
\[
H^2(L^\bullet(\mathcal{F}))=\mathrm{Ext}^2(\mathcal{F}, \mathcal{F})\simeq \mathrm{Hom}(\mathcal{F}, \mathcal{F}\otimes \mathcal{O}_{\mathbb{P}^2}(-3))^\vee=0,
\]
where the last equality again results from stability of $\mathcal{F}$. It follows that
\[
 \mathcal{T}_M= - {R}q_*{R}\mathcal{H} \kern -1.2pt \mathit{om}(\mathbb{F}, \mathbb{F}) +\CO_{M_{4,1}} \in K(M_{4,1}).
\]
Taking Chern characters, we obtain
\[
\mathrm{ch}(\mathcal{T}_M)=-\mathrm{ch}(q_*( \mathbb{F}^\vee \otimes^{L} \mathbb{F}))+1.
\]
We calculate using Grothendieck--Riemann--Roch that
\begin{align*}
\mathrm{ch}(q_*( \mathbb{F}^\vee \otimes^{L} \mathbb{F}))&=q_*(\mathrm{ch}( \mathbb{F}^\vee \otimes^{L} \mathbb{F})\cdot \mathrm{td}(\mathbb{P}^2))\\&= q_* (\mathrm{ch}(\mathbb{F}^\vee)\cdot \mathrm{ch}(\mathbb{F})\cdot \mathrm{td}(\mathbb{P}^2))\\&=
q_* (\mathrm{ch}^{\delta_0}(\mathbb{F}^\vee)\cdot \mathrm{ch}^{\delta_0}(\mathbb{F})\cdot \mathrm{td}(\mathbb{P}^2)),
\end{align*}
where we write in the last term
\[
\mathrm{ch}^{\delta_0}(\mathbb{F}^\vee):=\mathrm{ch}(\mathbb{F}^\vee)\cdot \exp(-\delta_0)=\sum_{k\geq 1} (-1)^k \mathrm{ch}^{\delta_0}_k (\mathbb{F}).
\]
Therefore, we get
\begin{align*}
\mathrm{ch}(\mathcal{T}_M)&=-q_* (\mathrm{ch}^{\delta_0}(\mathbb{F}^\vee)\cdot \mathrm{ch}^{\delta_0}(\mathbb{F})\cdot \mathrm{td}(\mathbb{P}^2))+1\\&=17+12 c_0(2)+(c_0(2)^2+8c_2(1)+2c_0(2)c_2(0))+(12c_2(2)+3c_0(2)c_2(1))+ \cdots,
\end{align*}
where the omitted terms have algebraic degrees $\geq 4$. Hence the total Chern class is
\small{\begin{equation}\label{total}
c(\mathcal{T}_M)=1+ 12 \alpha +(71\alpha^2 -8c_2(1) -2\alpha c_2(0)) + (24 c_2(2)-90 \alpha c_2(1)-24 \alpha^2 c_2(0)+276 \alpha^3)+\cdots.
\end{equation}}

\smallskip

\normalsize
So far, we have already written $c(\mathcal{T}_M)$ in terms of $c_k(j)$. It remains to express $c_2(2)$ in terms of the five tautological generators (\ref{5generators}). This is achieved by an explicit computation using \cite[Proposition 2.6]{PS}. We state a modified version for tautological classes on $M_{4,1}$ here; see \textit{loc. cit.} for the setup and notation.

\begin{prop}
For every $\ell \geq 5$ and $n\in \{1,2,3\}$, the following identity holds for $M_{4,1}$:
\begin{equation}
 \label{relation}
 \sum_\mathbf{m} \prod_{s=1}^\ell \frac{((s-1)!)^{m_s}}{(m_s)!}\left(\sum_{\substack{i\geq 0}} (-1)^i \frac{\pi_M^*\gamma^i}{i!} \pi_M^* A_{s-i}- \frac{(\pi_R^*\beta+\pi_M^*\gamma)^i}{i!}(-1)^i \pi_M^* B_{s-i}\right)^{m_s}=0.
 \end{equation}
 Here, the first sum is over all $\ell$-tuple of non-negative integers $\mathbf{m}=(m_1,m_2,\ldots, m_\ell)$ such that $m_1+2m_2+\cdots +\ell m_\ell=\ell$, and writing $\widetilde{c}_s(j):=(-1)^{s+1}c_s(j)$, the terms $A_s, B_s$ are given by
\begin{align*}
    A_s &:=\widetilde{c}_{s+1}(0)+\left(\frac{11}{4}-n\right)\widetilde{c}_{s}(1)+\left(\frac{1}{2}n^2-\frac{11}{4}n+\frac{117}{32}\right)\widetilde{c}_{s-1}(2)\in H^{2s}({M_{4,1}}, \mathbb{Q}),\\
    B_s &:=\widetilde{c}_{s+1}(0)+\left(\frac{7}{4}-n\right)\widetilde{c}_{s}(1)+\left(\frac{1}{2}n^2-\frac{7}{4}n+\frac{45}{32}\right)\widetilde{c}_{s-1}(2)\in H^{2s}(M_{4,1}, \mathbb{Q}).
\end{align*}
\end{prop}

\begin{proof}
The proof follows almost line by line as in \cite[Proposition 2.6]{PS}, except that we use the twisted Chern character
$\mathrm{ch}^{\delta_0}(\mathbb{F}^\vee)=\mathrm{ch}(\mathbb{F}^\vee)\cdot \exp(-\delta_0)$ by the class $\delta_0$ given in Lemma \ref{lemma2.5}. The classes $\widetilde{c}_s(j)$ show up since we take the Chern character of the dual universal sheaf.
\end{proof}

\smallskip

Now we take $\ell=5$ and integrate (\ref{relation}) with respect to $\pi_R^* (\mathbf{1}_{\mathbb{P}^2})$ as in \cite[Section 2.3]{PS}. This leads to a relation in $H^6(M_{4,1},\mathbb{Q})$ of the form
\begin{equation}\label{H6}
C_1 + C_2 \gamma +C_3 \gamma^2 + C_4 \gamma^3=0,
\end{equation}
where $C_2, C_3$ are expressions entirely in terms of the tautological generators (\ref{5generators}), and $C_4\in \mathbb{Q}$ is a constant. The vanishing holds for \textit{every} universal sheaf $\mathbb{F}$ and the class $\gamma$ obtained from the normalization class $\delta_0$ associated with $\mathbb{F}$. In particular, if we write $\mathcal{L}$ to be the line bundle on $M_{4,1}$ corresponding to the divisor $c_0(2)$ and replace $\mathbb{F}$ with 
\[
\mathbb{F}':= \mathbb{F}\otimes q^* \mathcal{L}^{\otimes m}
\]
for $m\in \mathbb{Z}$, a straightforward computation shows that
\[
\gamma'=\gamma+ {m}\cdot c_0(2).
\]
Since (\ref{H6}) holds for all these $\gamma$, we deduce that $C_2=C_3=C_4=0$, using the fact that there are no relations among the tautological generators in $H^{*\leq 6}(M_{4,1},\mathbb{Q})$, \textit{c.f.} \cite[Section 3.2]{PS}. Therefore, we can actually set $\gamma=0$ in (\ref{relation}) for the computation. Setting $n=1,2,3$, we obtain as in \cite[Section 2.3]{PS} three linearly independent relations in $H^{6}(M_{4,1},\mathbb{Q})$, whose linear combinations give the following relations:
\begin{align*}
c_2(2)=&\; \frac{32}{3}c_3(0)c_2(0)-\frac{28}{3} c_3(0)c_0(2)-\frac{1}{4}c_2(1)c_0(2) -4c_1(2)c_2(0) +\frac{93}{32} c_1(2)c_0(2) \\&+ \frac{3}{8}c_2(0) c_0(2)^2+\frac{47}{768}c_0(2)^3.\\
    c_3(1)=&\; -4c_3(0)c_2(0) +\frac{35}{8}c_3(0)c_0(2)+\frac{15}{8} c_1(2) c_2(0) -\frac{405}{256}c_1(2)c_0(2)-\frac{55}{512}c_0(2)^3.\\
    c_4(0)=&\;\frac{7}{3}c_3(0)c_2(0)-\frac{49}{24}c_3(0) c_0(2)  +\frac{1}{12}c_2(1)c_2(0) -\frac{1}{24}c_2(1) c_0(2) -\frac{7}{8}c_1(2)c_2(0)\\&+\frac{691}{1024}c_1(2) c_0(2) +\frac{1}{48}c_2(0)^2 c_0(2)+\frac{1}{32}c_0(2)^2 c_2(0)+\frac{537}{16384}c_0(2)^3.
\end{align*}

\smallskip

This provides the desired expression of $c_2(2)$ in terms of (\ref{5generators}). Consequently, we obtain from (\ref{total}) an expression of $c(\mathcal{T}_M)$ in terms of (\ref{5generators}).
\normalsize
\subsubsection{The classes $x$ and $y$.}\label{Sec2.3.4}

We will use the following result in \cite{CM}, which expresses $c(\mathcal{T}_M)$ in terms of the generators $\alpha, \beta, x,y,z$.

\begin{prop}[{\cite[Proposition 7.5]{CM}}]\label{prop2.8}
The total Chern class of $M_{4,1}$ is given by
\begin{align*}
c(\mathcal{T}_M) &= 1 + 12\alpha + (66\alpha^2 -3 \alpha \beta -3 \beta^2 +6x +2y +34z)+(220\alpha^3 -33\alpha^2 \beta -33\alpha \beta^2 -4\beta^3 \\& + 60 \alpha x -6\beta x + 30 \alpha y+ 22 \beta y +414 \alpha z + 22\beta z) + \textnormal{terms of algebraic degrees $\geq 4$.}
\end{align*}
\end{prop}

Comparing the terms of algebraic degrees $2$ and $3$ in this expression with (\ref{total}), we obtain
\begin{align*}
    &71\alpha^2 -8c_2(1) -2\alpha c_2(0)=66\alpha^2 -3 \alpha \beta -3 \beta^2 +6x +2y +34z,\\& 24 c_2(2)-90 \alpha c_2(1)-24 \alpha^2 c_2(0)+276 \alpha^3=220\alpha^3 -33\alpha^2 \beta -33\alpha \beta^2-4\beta^3 + 60 \alpha x -6\beta x \\&\kern 18.5 em + 30 \alpha y+ 22 \beta y +414 \alpha z + 22\beta z.
\end{align*}

Except for the classes $x$ and $y$, everything in the two identities is known in terms of (\ref{5generators}). Recall that there is no relation among the tautological generators in $H^{*\leq 6}(M_{4,1},\mathbb{Q})$, we obtain the expressions for $x$ and $y$ as we want:
\begin{align*}
x=&\;4 c_3(0) -\frac{1}{8} c_1(2)+4 c_2(0)^2- 8 c_0(2) c_2(0) +  \frac{831}{256}c_0(2)^2,
    \\
    y =&\; 5 c_3(0) +\frac{1}{4} c_2(1) -\frac{39}{32} c_1(2) +\frac{7}{2} c_2(0)^2 -\frac{221}{32}c_0(2)c_2(0)+\frac{5423}{2048}c_0(2)^2.
\end{align*}

\smallskip

This completes the comparison between the tautological generators (\ref{5generators}) and the generators $\alpha, \beta, x, y, z$ in Theorem \ref{thm2.1}. We summarize our results in the following theorem; combined with Theorem \ref{thm2.1}, we are able to write the cohomology $H^*(M_{4,1}, \BQ)$ in terms of the generators (\ref{5generators}) and relations.

\begin{thm}\label{prop2.9}
The following identities hold in $H^{*}(M_{4,1}, \mathbb{Q})$:
\begin{align*}
c_0(2) &= \alpha,\; \;
c_2(0) = \frac{25}{32}\alpha+\frac{1}{4}\beta,\\
c_1(2) &= \frac{3}{16} \alpha^{2}+\frac{1}{4} \alpha \beta +x -y -z,\\
c_2(1) &= \frac{55}{128} \alpha^{2}+\frac{5}{16} \alpha \beta +\frac{3}{8} \beta^{2} -\frac{3}{4} x -\frac{1}{4} y -\frac{17}{4} z,\\
c_3(0) &= \frac{75}{512} \alpha^{2}+\frac{15}{128} \alpha\beta  -\frac{1}{16} \beta^{2}+\frac{9}{32} x -\frac{1}{32} y -\frac{1}{32} z.
\end{align*}
\end{thm}

\smallskip

\subsection{Proof of the main results for degree $4$.} 

In this section we prove Theorem \ref{Thm0.6} and Theorem \ref{thm0.7} for $d=4$. Consider now the morphism $h: M_{4,1} \to \BP^{14}$. By Corollary \ref{cor1.3}, the perverse filtration can be characterized using the class $\xi = c_0(2)$ as follows:
\begin{equation}\label{perv0}
P_k H^m(M_{4,1},\mathbb{Q})=\sum_{i\geq 1} \big( \mathrm{Ker} (\xi^{14+k-m+i}) \cap \mathrm{Im} (\xi^{i-1})\big) \cap H^m(M_{4,1}, \mathbb{Q}).
\end{equation}
The equation (\ref{perv0}), combined with the ring structure of $H^*(M_{4,1}, \BQ)$ given by Theorem \ref{thm2.1} and Theorem \ref{prop2.9}, provides a complete description of the perverse filtration in terms of the generators in Theorem \ref{thm2.1} or the five tautological generators (\ref{5generators}). In particular, using a computer, we are able to check\footnote{The computation is conducted via the software \textsc{Macaulay2.} Explicit descriptions of the perverse filtration for $M_{4,1}$ can be found on the second author's website: \href{https://github.com/Weite-Pi/weitepi.github.io/blob/6e08a17f32af07e6944eba88ea105e09d448aa64/PM41.pdf}{\texttt{https://github.com/Weite-Pi/weitepi.github.io}.}} 
\[
P_kH^{* \leq 4} (M_{4,1}, \BQ) = C_kH^{*\leq 4}(M_{4,1}, \BQ),
\]
and calculate the dimensions of the graded piece
\[
\dim \mathrm{Gr}_k^P H^m(M_{4,1}, \BQ), \quad \forall k,m\in \BZ.
\]
The refined invariants
\[
n_4^{i,j} = \dim \mathrm{Gr}^P_iH^{i+j}(M_{4,1}, \BQ)
\]
are matched with the formula $\widetilde{F}_{\mathrm{BPS},4}(q,t)$ in Section \ref{Section 4} obtained via the Nekrasov partition function. This proves both Theorem \ref{Thm0.6} and Theorem \ref{thm0.7}.

In the following, we provide more details on checking $P=C$ for the reader's convenience.
The filtration $P_\bullet H^m(M_{4,1}, \mathbb{Q})$ is concentrated in perverse degrees $[0,m]$, and the fundamental class $\mathbf{1}_{M_{4,1}}$ lies in $P_0 H^0(M_{4,1}, \mathbb{Q})$. Since the class $c_0(2)$ is the pull-back of a hyperplane class on the base $\mathbb{P}^{14}$, and $c_2(0)$ is a relative ample class \cite{PS}, we obtain immediately from relative Hard Lefschetz that:

\begin{center}
\begin{tabular}{||c c c c c c||} 
 \hline
   & $c_0(2)$ & $c_2(0)$ & $c_0(2)^2$ & $c_0(2)c_2(0)$ & $c_2(0)^2$\\ 
 \hline
 Perversity & 0 & 2 & 0 & 2 & 4  \\
 \hline
\end{tabular}
\end{center}

Furthermore, we use (\ref{perv0}), Theorem \ref{thm2.1}, and Theorem \ref{prop2.9} to check that the perversity of the generators in $H^4(M_{4,1}, \mathbb{Q})$ are as expected:

\begin{center}
\begin{tabular}{||c c c c ||} 
 \hline
   & $c_1(2)$ & $c_2(1)$ & $c_3(0)$ \\ 
 \hline
 Perversity & 1 & 2 & 3  \\
 \hline
\end{tabular}
\end{center}

Finally, we check as above that any $\BQ$-linear combination of the classes
\[
c_0(2)c_2(0), ~ c_2(1) \in H^4(M_{4,1}, \BQ)
\]
has perversity $2$. In particular, we have
\[
\mathrm{Span}_\mathbb{Q}\langle c_0(2) c_2(0), c_2(1) \rangle \cap P_1 H^4(M_{4,1}, \mathbb{Q})=\{0\}.
\]
This guarantees the identity
\[
P_kH^{* \leq 4} (M_{4,1}, \BQ) = C_kH^{*\leq 4}(M_{4,1}, \BQ)
\]
which completes the proof. \qed


\begin{rmk}\label{rmk2.11}
One can verify also that the class $c_2(0)c_3(0)$ has perversity $3$, so we see that the bound $2d-4$ in Conjecture \ref{P=C} is optimal for $d=4$.
\end{rmk}



\section{Pandharipande--Thomas theory, Nekrasov partition functions, and combinatorial BPS invariants}
\label{Section 3}

\subsection{Overview}

In this section, we introduce the \emph{combinatorial} BPS invariants
\[
\widetilde{F}_{d,\mathrm{BPS}} (q,t) = \sum_{i,j} \widetilde{n}_d^{i,j} q^it^j.
\]
They are defined by the Nekrasov partition function (\ref{comb_BPS}), and they refine the (standard) Pandharipande--Thomas (PT) invariants \cite{PT2,PT1,PT3} for the local $\BP^2$.

In contrast to $n_d^{i,j}$ defined via the perverse filtration, the combinatorial invariants $\widetilde{n}_d^{i,j}$ are very easy to compute. For a fixed $d$, the generating function $\widetilde{F}_{d,\mathrm{BPS}}(q,t)$ is obtained by a calculation in only finitely many terms. Nevertheless, these two types of invariants are expected to coincide by string-theoretic considerations.


\begin{conj}[\emph{c.f.} \cite{NO}]\label{conj0}
We have
\[
\widetilde{n}_{d}^{i,j} = n_d^{i,j}.
\]
\end{conj}

Conjecture \ref{conj0} is the main source for us to make predictions on the structure of the refined BPS invariants $n^{i,j}_d$. For example, using a computer we are able to check Conjecture \ref{conj1} for $\widetilde{n}_d^{i,j}$ in all degrees $d\leq 14$. 

\smallskip

Finally, we list in Section \ref{Section 4} the formulas \[
\widetilde{F}_{d,\mathrm{BPS}}(q,t), \quad \textrm{for } d=3, 4 
\]
which are needed to match the invariants (\ref{BPS1}) obtained from the perverse filtration. These formulas were also obtained in \cite{CKK}.

\subsection{PT theory of local $\BP^2$}

We first recall the (unrefined) PT invariants for $X =\mathrm{Tot}(K_{\mathbb{P}^2})$. We consider the moduli space $\mathrm{PT}(X,d,n)$ of stable pairs $(\CF,s)$ where $\CF$ is a pure 1-dimensional sheaf on $X$ with
\[
[\mathrm{supp}(\CF)] = dH \in H_2(\BP^2, \BZ) =  H_2(X,\BZ), \quad \chi(\CF) = n
\]
and a section $s: \mathcal{O}_X \to \mathcal{F}$ satisfying that $\dim \mathrm{coker} \ s = 0$. Although $X$ is not projective itself, the moduli spaces $\mathrm{PT}(X,d,n)$ are projective. The PT invariants \cite{PT2} are defined to be the degrees of the virtual cycles
\[
\mathrm{PT}_{n,d} : = \int_{[\mathrm{PT}(X,d,n)]^\mathrm{vir}}1 \in \BZ.
\]
They form a generating function
\begin{equation}\label{PT0}
Z^{\mathrm{ur}}_{\mathrm{PT}}(X) : = \sum_{d} Q^d
\sum_{n \in \mathbb{Z}} \mathrm{PT}_{n,d} \cdot (-z)^n.  
\end{equation}
Here the superscript $``\mathrm{ur}"$ stands for {unrefined} invariants.

\subsection{Refined PT invariants}

In order to refine (\ref{PT0}), we consider the 3-dimensional torus $\mathbb{C}^\times_{q_1, q_2, q_3}$ acting on ${X}$ such that
\[
\mathbb{C}^\times_{q_1, q_2, q_3} = \left\{
\begin{pmatrix}
q_1^{-1} &  &\\
& q_2^{-1} & \\
& & q_3^{-1}
\end{pmatrix}
\right\} \subset \mathrm{Aut}\left(H^0(\BP^2, \mathcal{O}_{\mathbb{P}^2}(1))\right)
\]
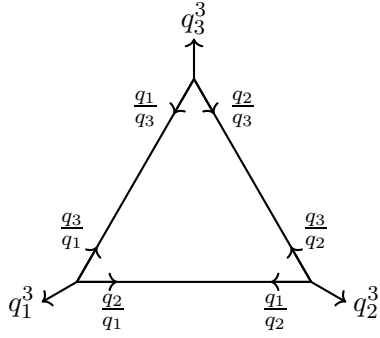
\begin{figure}[h]
    \centering
    \begin{tikzpicture}[scale=0.027]
    \draw [thick] (-57.7,0)--(57.7,0)--(0,100)--cycle;
    \draw [thick, ->] (-57.7,0) -- (-37.7,0);
    \draw [thick, ->] (57.7,0) -- (37.7,0);
    \draw [thick, ->] (-57.7,0) -- (-47.7, 17.3);
    \draw [thick, ->] (57.7,0) -- (47.7, 17.3);
    \draw [thick, ->] (0,100) -- (-10,82.7);
    \draw [thick, ->] (0,100) -- (10,82.7);
    \draw [thick, ->] (0,100) -- (0,120);
    \draw [thick, ->] (-57.7,0) -- (-75, -10);
    \draw [thick, ->] (57.7,0) -- (75, -10);
    \node at (-85, -10) {$q_1^3$};
    \node at (85, -10) {$q_2^3$};
    \node at (0, 130) {$q_3^3$};

    \node at (-40, -15) {$\frac{q_2}{q_1}$};
    \node at (40, -15) {$\frac{q_1}{q_2}$};
    \node at (-60, 25) {$\frac{q_3}{q_1}$};
    \node at (60, 25) {$\frac{q_3}{q_2}$};
    \node at (-24, 85) {$\frac{q_1}{q_3}$};
    \node at (24, 85) {$\frac{q_2}{q_3}$};
    \end{tikzpicture}
    \caption{Toric diagram of $\BP^2$. Vertices denote the fixed points, and the edges denote the invariant rational curves. Arrows at the vertices denote the weights of the tangent spaces.}
    \label{ToricP2}
\end{figure}

The perfect obstruction theory on $\mathrm{PT}(X,d,n)$ yields a virtual structure sheaf $\mathcal{O}^\mathrm{vir}$ and its symmetrized version 
\[
\widehat{\mathcal{O}}^\mathrm{vir} := \mathcal{O}^\mathrm{vir} \otimes \left(\mathcal{K}^\mathrm{vir}\right)^{\frac{1}{2}},
\]
as explained in \cite{OkPCMI}. The virtual index with respect to $\widehat{\mathcal{O}}^\mathrm{vir}$ coincides with the invariant $\mathrm{PT}_{n,d}$:
\begin{equation}\label{PT=chi}
\chi(\mathrm{PT}(X, d, n), \widehat{\mathcal{O}}^\mathrm{vir}) = \mathrm{PT}_{n,d} \in \BZ;
\end{equation}
see \cite[(7.7)]{CKK} and \cite{NO}. Then the \emph{equivariant} virtual index with respect to the torus action naturally refines $\mathrm{PT}_{n,d}$. More precisely, in view of (\ref{PT=chi}), we express (\ref{PT0}) as the $K$-theoretic PT generating function of virtual indices:
\begin{equation}\label{KPT}
Z^{\mathrm{ur}}_{\mathrm{PT}}({X}) = \sum_{d} Q^\beta 
\sum_{n \in \mathbb{Z}} \chi(\mathrm{PT}(X, d, n), \widehat{\mathcal{O}}^\mathrm{vir})\cdot (-z)^n. 
\end{equation}
Using the torus action $\mathbb{C}^\times_{q_1, q_2, q_3}$ on $X$ and localization, the equivariant version $Z^{\mathrm{equiv}}_{\mathrm{PT}}(X)$ of the $K$-theoretic generating function (\ref{KPT}) is expressed as the contraction of three equivariant PT vertex functions along the edges of the toric diagram. The localization formula \cite{OkT, OkPCMI} relies only on the description of the tangent space to a pair $(\mathcal{F},s)$ as 
\[
\text{Tangent space at $(\mathcal{F},s)$} = \chi(\mathcal{F}) + \chi(\mathcal{F},\mathcal{O}) - \chi(\mathcal{F},\mathcal{F}).
\]
It turns out that each coefficient in front of $Q^d$ converges to a rational function in the variable $z$. The coefficients of the expansion of the plethystic logarithm of $Z^{\mathrm{equiv}}_{\mathrm{PT}}(X)$ are the Gopakumar--Vafa invariants $\mathrm{GV}_d \in \BZ[z, \kappa^{\frac{1}{2}}]$:
\[
Z^{\mathrm{equiv}}_{\mathrm{PT}}(X) = \mathsf{S}^\bullet \left(
\sum_{d \geq 1} \frac{\mathrm{GV}_d}{(1-z \sqrt \kappa)(1-\frac{z}{\sqrt \kappa})}\cdot Q^d
\right), \quad  \kappa = q_1q_2q_3.
\]
Each $\mathrm{GV}_d$ is equivalent to $\widetilde{F}_{d,\mathrm{BPS}}(q,t)$ up to a change of variables (see (\ref{comb_BPS}) below), and we list the first few expressions of $\mathrm{GV}_d$:
\begin{align*}
\mathrm{GV}_1 &= z \cdot (\kappa +1 + \kappa^{-1}),\\
\mathrm{GV}_2 &= 
-z\cdot (\kappa^{\frac{5}{2}}+\kappa^{\frac{3}{2}}+\kappa^{\frac{1}{2}}+\kappa^{-\frac{1}{2}}+\kappa^{-\frac{3}{2}}+\kappa^{-\frac{5}{2}}),\\
\mathrm{GV}_3 &= z\cdot (\kappa^{3}+\kappa^{2}+\kappa+1+\kappa^{-1}+\kappa^{-2}+\kappa^{-3})\\&\;+ (z^2+1)\cdot (\kappa^{\frac{9}{2}}+\kappa^{\frac{7}{2}}+\kappa^{\frac{5}{2}}+\kappa^{\frac{3}{2}}+\kappa^{\frac{1}{2}}+\kappa^{-\frac{1}{2}}+\kappa^{-\frac{3}{2}}+\kappa^{-\frac{5}{2}}+\kappa^{-\frac{7}{2}}+\kappa^{-\frac{9}{2}}).
\end{align*}
Indeed, it was proven in \cite[Theorem 1]{NO} that $Z^{\mathrm{equiv}}_{\mathrm{PT}}(X)$ depends on the equivariant parameters $q_i$ with $i=1,2,3$ only as a function of the weight  of the Calabi--Yau 3-form $\kappa = q_1 q_2 q_3$. Thus it can be evaluated in the refined limit $q_i \to 0, \infty$ when  $\kappa = \mathrm{const}$.

The refined limit of the vertex function is particularly simple when there is a {\it preferred} direction. Briefly speaking, the preferred direction is a weight of $\BC^3$ which goes to $0$ or $\infty$ much more slowly than the other two weights. In the case of the Donaldson--Thomas vertex it was  investigated in \cite{NO}, and for the 2-legged PT vertex  in \cite{KOO}.

\smallskip

Unfortunately, for local $\mathbb{P}^2$ we need to consider the 3-legged PT vertex, and it is not possible to make a limit of equivariant parameters such that each of the three vertices has a preferred direction. For example, if we take a limit
\[
q_1 \approx q_2 \to 0, \ q_3 \to \infty, \ q_1 q_2 q_3 = \mathrm{const},
\]
then the bottom two vertices have preferred directions along the edge joining them, while the third vertex does not have one.
Thus, we can not reduce $Z^{\mathrm{equiv}}_{\mathrm{PT}}(X)$ completely to the refined topological vertex \cite{IQV}, and it requires a new type of vertex as explained in \cite{IQ}.

Another approach is to relate, using the flop relation, $Z^{\mathrm{equiv}}_{\mathrm{PT}}(X)$ with the equivariant PT invariants of the local $A_1$-surface, for which there is a limit such that each vertex has a preferred direction. As a consequence, such calculation relates $Z^{\mathrm{equiv}}_{\mathrm{PT}}(X)$ to the Nekrasov partition function associated with the rank 2 instanton moduli space with an insertion of the tautological line bundle $\mathcal{O}(1)$ given by (\ref{V}). More generally, equivariant PT counts for the local $A_n$-surface are related to the Euler characteristic of the sheaf $\mathcal{O}(n)$ on the instanton moduli space. In this case, the two K\"ahler parameters for the PT moduli space are related with one K\"ahler and one equivariant parameters of the framing for the instanton moduli space.
By this approach, we may recover $\mathrm{GV}_d$ from the instanton invariants $P_d$ defined later in Section \ref{Sec3.4}:
\begin{equation}\label{Pd}
\mathrm{GV}_d (z, \kappa) = \mathrm{monomial \ prefactor} \cdot P_d(z \sqrt{\kappa}, \frac{z}{\sqrt{\kappa}}).
\end{equation}

\subsection{Nekrasov partition functions}\label{Sec3.4}

Let $M(r,n)$ be the  Nakajima  variety \cite{NakajimaBook} corresponding to the quiver as in Figure \ref{Nak}.
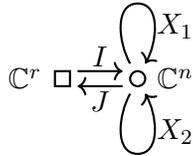
\begin{figure}[h]
    \centering
    \begin{tikzpicture}
    \draw [thick] (-0.1,-0.1)--(-0.1,0.1)--(0.1,0.1)--(0.1,-0.1)--cycle;
    \draw [thick] (1,0) circle (0.1);
    \draw [->, thick, rotate around = {90: (1,0)}] (1.2,0.1) to [out=30, in=90] (2,0) to [out=270, in=330] (1.2,-0.1);
    \draw [<-, thick, rotate around = {-90: (1,0)}] (1.2,0.1) to [out=30, in=90] (2,0) to [out=270, in=330] (1.2,-0.1);
    \draw [->, thick] (0.2,0.1) -- (0.8,0.1);
    \draw [<-, thick] (0.2,-0.1) -- (0.8,-0.1);
    \node at (-0.5,0) {$\mathbb{C}^r$};
    \node at (1.5,0) {$\mathbb{C}^n$};
    \node at (0.5,0.3) {$I$};
    \node at (0.5,-0.3) {$J$};
    \node at (1.5,0.7) {$X_1$};
    \node at (1.5,-0.7) {$X_2$};
    \end{tikzpicture}
    \caption{The graph consists of two vertices, $\BC^n$ and $\BC^r$. The latter is referred to as the framing vertex.
    The variety is defined as the GIT quotient
    $M(r,n) := \{X_1, X_2, I,J \mid [X_1, X_2] + I\cdot J = 0\}\sslash \mathrm{GL}(n)$}
    \label{Nak}
\end{figure}

In the physics literature $M(r,n)$ is known as the instanton moduli space. It is isomorphic to the moduli space of rank $r$ framed torsion-free sheaves $\mathcal{F}$ on $\mathbb{P}^2$, such that $c_2(\mathcal{F}) = n$,  with a trivialization over $\mathbb{P}_\infty^1 \subset \mathbb{P}^2$:
\[
\mathcal{F}\Big|_{\mathbb{P}^1} \cong \mathcal{O}_{\BP^1}^{\oplus r}.
\] 
Note that the $\BP^2$ here has nothing to do with the $\BP^2$ for the PT theory; they are complement to each other from the M-theoretic point of view.

\smallskip

There is an action of the torus
\[
\mathsf{T} = \mathbb{C}^\times_{t_1, t_2} \times \mathbb{C}^\times_{u_1,...,u_n},
\]
on $M(r,n)$, where $\mathbb{C}^\times_{t_1,t_2}$ acts
as
\[
X_i \mapsto t_i X_i, \ \ I \mapsto I, \ \ J \mapsto t_1 t_2 J,
\]
and 
\[\mathbb{C}^\times_{u_1,...,u_r} = \left\{
\begin{pmatrix}
u_1 &  &\\
& \ddots & \\
& & u_r
\end{pmatrix}
\right\} \subset \mathrm{Aut}(\mathbb{C}^r).
\]

\smallskip

The Nakajima variety $M(r,n)$ is smooth and of dimension $2rn$.
The vertex $\mathbb{C}^n$ in the quiver gives rise to the tautological vector bundle $\mathcal{V}$ of rank $n$ over $M(r,n)$, and the framing vertex gives rise to the trivial rank $r$ vector bundle $\mathcal{W}$ with the character of the fiber $u_1+u_2+\cdots+u_r$.  The $K$-theory class of the tangent bundle to $M(r,n)$ can be written as
\begin{equation}\label{tanHilb}
\mathcal{T} = \Hom(\mathcal{W},\mathcal{V}) + t_1t_2 \Hom(\mathcal{V},\mathcal{W})-(1-t_1)(1-t_2) \Hom(\mathcal{V},\mathcal{V}).
\end{equation}
The Picard group of $M(r,n)$ is of rank 1 generated by the determinant bundle
\begin{equation}\label{V}
\mathcal{O}(1): = \det \mathcal{V}.
\end{equation}
See \cite{NakajimaBook} for details on these facts. Finally, the Nekrasov partition function is defined as the the generating series for the equivariant Euler characteristics:
\[
Z(\ell) = \sum_{n\geq 0} z^n \chi_\mathsf{T} (M(r,n), \mathcal{O}(\ell)) \in  \mathbb{Q}(t_1,t_2,u_1,...,u_r)[[z]].
\]
The function $Z(\ell)$ can be computed explicitly by equivariant localization. The set of fixed points $M(r,n)^\mathsf{T}$ are identified with $r$-tuples of Young diagrams with $n$ boxes in total. The contribution of each fixed point can be easily computed using the formula (\ref{tanHilb}).

\smallskip

Now the PT invariants for $X = \mathrm{Tot}({K_{\mathbb{P}^2}})$ can be obtained as a certain limit for $Z(1)$ with $r=2$, as we explain in the following. The partition function $Z(\ell)$ depends on the four variables $z,t_1,t_2,{u_2}/{u_1}$. Setting $u={u_2}/{u_1}$,
the plethystic logarithm of $Z(1)$ then determines the invariants $P_d(t_1,t_2)$:
\small{\[
Z(1) = {\mathsf{S}}^\bullet \left(
\frac{t_1^2 t_2^2}{(1-t_1)(1-t_2)}u\cdot z
+
\sum_{d \geq 1} z^d \left( (-1)^{d-1} (t_1t_2)^{\frac{-d(d-3)}{2}}\cdot \frac{P_d(t_1,t_2)}{(1-t_1)(1-t_2)} u^{2d} + O(u^{2d+1})
\right)
\right).
\]}
\normalsize

Each $P_d(t_1,t_2)$, mentioned earlier in (\ref{Pd}), is a  symmetric polynomial in the variables $t_1, t_2$. We define the combinatorial BPS invariants by
\begin{equation}\label{comb_BPS}
\widetilde{F}_{d,\mathrm{BPS}}(q,t) := \left. \left( t_1^{-\frac{(d-1)(d-2)}{2}} \cdot P_d(t_1, t_2)
\right) \right|_{t_1 = \frac{t}{q}, t_2 = tq} \in 1 + t^2 \mathbb{Z}[t,q].
\end{equation}

\begin{rmk}
We see from the formulas of $Z(1)$ and $P_d$ that the generating series $\widetilde{F}_{d,\mathrm{BPS}}(q,t)$ can be evaluated by a calculation of finitely many terms. We also note that $\widetilde{F}_{d,\mathrm{BPS}}(q,t)$ is equivalent to the Gopakumar--Vafa invariants $\mathrm{GV}_d$ via $P_d(t_1,t_2)$. Therefore it recovers and refines the PT invariants (\ref{PT0}).
\end{rmk}

\subsection{Numerical data}\label{Section 4}
We list in this last section the combinatorial BPS invariants for $d=3,4$, which are used in the proof of Theorem \ref{thm0.7}. The terms that coincide with the expansion of $H(q,t)$, \textit{c.f.} the remark after Conjecture \ref{conj0}, are enclosed in a square bracket.

{
\small

\vspace{0pt}
\begin{align*}
\kern -40pt \widetilde{F}_{3,\mathrm{BPS}} = & \; {\boldsymbol{\left[{ 1+\left(q^{2}+t^{2}\right)}\right]}}+\left(q^{2} t^{2}+q t^{3}+t^{4}\right) \\&+\left(q^{2} t^{4}+q t^{5}+t^{6}\right) +\left(q^{2} t^{6}+q t^{7}+t^{8}\right) 
\\&+
\left(q^{2} t^{8}+q t^{9}+t^{10}\right)
+\left(q^{2} t^{10}+q t^{11}+t^{12}\right)\\& +\left(q^{2} t^{12}+q t^{13}+t^{14}\right) 
+\left(q^{2} t^{14}+q t^{15}+t^{16}\right)\\ &+\left(q^{2} t^{16}+t^{18}\right) +q^{2} t^{18}.
\end{align*}

\vspace{-10pt}
\begin{align*}
\kern 13pt \widetilde{F}_{4,\mathrm{BPS}} = & \;{\boldsymbol{ \left[{1+\left(q^{2}+t^{2}\right) +\left(q^{4}+q^{3} t +2 q^{2} t^{2}+q t^{3}+t^{4}\right)}\right]}}\\&+\left(q^{6}+2 q^{4} t^{2}+2 q^{3} t^{3}+3 q^{2} t^{4}+q t^{5}+t^{6}\right)
\\&+\left(q^{6} t^{2}+q^{5} t^{3}+3 q^{4} t^{4}+3 q^{3} t^{5}+4 q^{2} t^{6}+q t^{7}+t^{8}\right)\\&+\left(q^{6} t^{4}+q^{5} t^{5}+4 q^{4}t^{6}+3 q^{3} t^{7}+4 q^{2} t^{8}+qt^{9}+t^{10}\right)
\\&+\left(q^{6} t^{6}+q^{5} t^{7}+4 q^{4} t^{8}+4 q^{3} t^{9}+4 q^{2} t^{10}+q t^{11}+t^{12}\right)\\&+\left(q^{6} t^{8}+q^{5} t^{9}+4 q^{4} t^{10}+4 q^{3} t^{11}+4 q^{2} t^{12}+q t^{13}+t^{14}\right) \\&+
\left(q^{6} t^{10}+q^{5} t^{11}+4 q^{4} t^{12}+4 q^{3} t^{13}+4 q^{2} t^{14}+q t^{15}+t^{16}\right)\\& +\left(q^{6} t^{12}+q^{5} t^{13}+4 q^{4} t^{14}+4 q^{3} t^{15}+4 q^{2} t^{16}+q t^{17}+t^{18}\right)\\&
+\left(q^{6} t^{14}+q^{5} t^{15}+4 q^{4} t^{16}+4 q^{3} t^{17}+4 q^{2} t^{18}+q t^{19}+t^{20}\right)\\&+\left(q^{6} t^{16}+q^{5} t^{17}+4 q^{4} t^{18}+4 q^{3} t^{19}+4 q^{2} t^{20}+q t^{21}+t^{22}\right) \\&
+\left(q^{6} t^{18}+q^{5} t^{19}+4 q^{4} t^{20}+3 q^{3} t^{21}+4 q^{2} t^{22}+q t^{23}+t^{24}\right)\\&+\left(q^{6} t^{20}+q^{5} t^{21}+4 q^{4} t^{22}+3 q^{3} t^{23}+3 q^{2} t^{24}+q t^{25}+t^{26}\right)\\&
+\left(q^{6} t^{22}+q^{5} t^{23}+3 q^{4} t^{24}+2 q^{3} t^{25}+2 q^{2} t^{26}+t^{28}\right)\\&+\left(q^{6} t^{24}+q^{5} t^{25}+2 q^{4} t^{26}+q^{3} t^{27}+q^{2} t^{28}\right)\\&+\left(q^{6} t^{26}+q^{4} t^{28}\right) +q^{6} t^{28}.
\end{align*}
}

\newpage \noindent By the same method, we have checked the combinatorial version of Conjecture \ref{conj1}
\[
\widetilde{n}_d^{i,j} = [H(q,t)]^{i,j}, \quad  i+j \leq 2d-4
\]
for all degrees $d\leq 14$.

\end{document}